\documentclass[12pt]{amsart}
\usepackage{standard}
\usepackage{tikz}
\usepackage{mathtools}
\usepackage{geometry} 
\usepackage{xcolor}

\DeclareMathOperator{\sgn}{sgn}

\usepackage{mathrsfs}

\newcommand{\E}{\mathscr{E}}
\newcommand{\D}{\mathscr{D}}
\newcommand{\schwartz}{\mathcal{S}}
\DeclareMathOperator{\proj}{proj}
\DeclareMathOperator{\elliptic}{ell}
\newcommand{\loc}{\text{loc}}

\newcommand{\hilb}{\mathcal{H}}
\theoremstyle{theorem}
\newtheorem*{maintheorem}{Theorem}

\newcommand{\C}{\mathsf{A}} 

\title[Wave equation on a rotating cosmic string background]{Mode solutions to the wave equation on a rotating cosmic string background}
\author{Katrina Morgan and Jared Wunsch}
\address{Department of Mathematics\\Northwestern University\\Evanston
  IL 60208\\USA}
\email{katrina.morgan@northwestern.edu}
\email{jwunsch@math.northwestern.edu}
\date{\today}

\begin{document}
\thanks{The authors are grateful to Angel Carrillo, Kevin Payne, and
  Andr\'as Vasy for helpful conversations, {as well as to an anonymous
  referee for improving the exposition}.  KM was partly supported by NSF Postdoctoral
  Fellowship DMS--2002132.   JW was partially supported
by Simons Foundation grant 631302, NSF grant DMS--2054424, and a
Simons Fellowship.}

\begin{abstract}
  A static rotating cosmic string metric is singular along a timelike
  line and fails to be globally hyperbolic; these features make it
  difficult to solve the wave equation by conventional energy methods.
  Working on a single angular mode at a time, we use microlocal
  methods to construct forward parametrices for wave and Klein--Gordon
equations on such backgrounds.
  \end{abstract}

\maketitle
\section{Introduction}
In this note we construct a semi-global forward parametrix for
mode solutions to the wave equation on a rotating cosmic string background.  Cosmic strings, introduced by Kibble
\cite{Ki:76}, are solutions to the Einstein equations that have
topological defects along one-dimensional (``string'') structures.
They may or may not be a feature of real cosmology
\cite{LI:21}.  The simplest cosmic string solutions, corresponding to
a single, nonrotating string in equilibrium, may be viewed
either as singular at $(x_1,x_2)=0,\ (x_3, t) \in \RR$ in $3+1$ dimensions or, reducing along
an axis of symmetry, as singular at $(x_1,x_2)=0,\ t \in \RR$ in $2+1$
dimensions.  The latter solutions are simply static metrics whose spatial slices are flat 2d
cones.  \emph{Rotating} cosmic string solutions, by
contrast, have a singularity with an authentically Lorentzian
character, given in the static setting \cite[Equation 4.17]{DeJaTh:84} by 
the metric (in cylindrical coordinates in $\RR_t \times \RR^2$)
\begin{equation}\label{metric}
g=(dr^2 + r^2 d\varphi^2)-(dt^2-2\C \, dt \, d\varphi+ \C^2d\varphi^2).
\end{equation}
{These are solutions, introduced by Deser--Jackiw--'t Hooft \cite{DeJaTh:84}, to the Einstein equations corresponding to a
one-dimensional rotating source with zero mass but with nonzero angular
  momentum; here $\C=-4GJ$ where $G$ is the gravitational constant and
  $J$ the angular momentum.}
Owing to their flatness, which can be seen \emph{locally} by a change of
coordinates reducing to Minkowski space, these metrics are manifestly singular solutions to the Einstein equations with
vanishing cosmological constant.
Among the interesting features of the rotating cosmic string metric
(dubbed a ``cosmon'' in \cite{DeJaTh:84}) are the singularity at
$r=0$ and the causality violation entailed by the existence of closed
timelike curves such as $$(t_0, r_0, \varphi=s)\colon s \in [0,
2\pi]$$
for $r_0 <\C.$

One might suppose that such a serious causality violation as exhibited
by the metric \eqref{metric} should be disastrous for the
well-posedness of the wave equation on such a background, and
certainly it does cause great difficulties for conventional energy
methods.   (Note that the conserved energy associated to
$t$-translation invariance ceases to be positive when $r<\C$.)
A number of positive results on the behavior of the wave equation on
causality-violating spacetimes similar to the one we study have been obtained in the
work of Bachelot \cite{Ba:02}, however.\footnote{Our metric
  essentially fits into
  the framework of \emph{Papepetrou metrics} considered by Bachelot,
  but the singularity at $r=0$ is a novel feature here, as is the
  focus on causal solutions, rather than scattering theory.}
  Motivated by \cite{Ba:02}, in this note we
pursue the question of existence of \emph{forward} solutions to the
wave equation, where we specify a compactly supported inhomogeneity,
and try to solve in forward time.  The wave operator in this spacetime
is given by
$$
  \big(1-\frac{\C^2}{r^2}\big)\pa_t^2 - \frac 1{r^{2}}
           (r\pa_r)^2 - \frac{2\C }{r^2}\pa_\varphi \pa_t-\frac{1}{r^2}\pa_\varphi^2.
$$
Here we specialize to
a single angular mode solution $e^{i k \varphi}u,$ which leaves us with the 1+1 dimensional operator
\begin{align*}
  \Box_k &= \big(1-\frac{\C^2}{r^2}\big)\pa_t^2 - \frac 1{r^{2}}
           (r\pa_r)^2 - \frac{2\C ik}{r^2} \pa_t+\frac{k^2}{r^2}\\
  &= -\frac 1{r^2} (\C\pa_t+ik)^2 +\pa_t^2 - \pa_r^2 - \frac 1r \pa_r.
\end{align*}
The operator $\Box_k$ thus changes type from hyperbolic in $r>\C$ to
elliptic in $r<\C.$ Near this interface, the equation is in fact of
\emph{Tricomi type}, with the added difficulty of a singularity at
$r=0.$ We seek a forward solution operator modulo smoothing terms.
One might hope for better: a solution to $\Box u=\delta_q$ \emph{supported} only in the
forward light cone emanating from a point $q$ in the hyperbolic
region, but this is ruled out by Lemma~\ref{lemma:nonull}
below (cf.\ \cite[Theorem 3.5]{Ba:02}): the support of the solution must
extend throughout the elliptic region.

Thus we rely on microlocal methods, in the spirit of the work of
Payne \cite{Pa:98} on equations of Tricomi type, which yield results
on singularities of solutions without actually constraining their
supports.  \emph{Our main result (fully stated in
  Section~\ref{section:function} below) is the existence of a forward parametrix for
the equation $(\Box_k+m^2) u=f:$} in particular, we show that there is an
exact solution to this equation whose wavefront set is contained in
the forward-in-time bicharacteristic flowout of $\WF f$ {inside the
characteristic set (i.e., the light cone), together with $\WF f$ itself.}  (Note that we
have additionally allowed a nonnegative mass term $m^2,$ and thus consider the
more general Klein--Gordon equation.)
By contrast we remark that the extensive treatment of solutions of the
Tricomi equation in the gas
dynamics literature (see e.g.\ \cite{Mo:04}) tends to emphasize
propagating data from either characteristic curves in the hyperbolic
region (Tricomi problem) or from noncharacteristic surfaces such as
would be locally given in our setup by $r=\text{constant}$ (Frankl's
problem).\footnote{It is claimed in
  \cite{BaNeGe:03} that there exists a fundamental solution to the
  Tricomi equation that is
  \emph{supported} in what the authors call Region
  III, which corresponds to our forward flowout, but the apparent
  contradiction with our results seems to be addressed by
 the erratum
  \cite{BaNeGe:03a}.}

We confess that dealing with single mode solutions, as we do here,
essentially sidesteps the worst difficulties of causality violation:
at high energy, the solutions we study have zero
angular momentum (the angular momentum $k$ is fixed, while the duals
to $t$ and $r$ become infinite). As we will see below, the
associated null-geodesic flow is thus well behaved and $t$ is monotone, with the slight caveat
that the null geodesics do have singularities at $r=\C.$ (The null
bicharacteristics, corresponding to lifts of the geodesics to the
cotangent bundle, remain nonsingular, in any event.)  We intend to treat the
full propagation of singularities for the wave equation on cosmic
string backgrounds (i.e., not just for mode solutions) in a subsequent
paper; this will entail a much more technical analysis of propagation of singularities
through the string at $r=0.$

\section{Function spaces and mapping properties}\label{section:function}

As we will be working on mode solutions, we could restrict our
attention to functions defined on the space $[0,\infty)_r \times
\RR_t,$ but this is potentially confusing owing to the artificial
boundary at $r=0$ and the volume form $r \,dr \, dt.$  Hence we
will instead deal explicitly with $k$-equivariant functions on $\RR^3 =\RR_t
\times \RR^2_x.$
To this end, we define adapted Sobolev spaces for our problem.  Let
$\norm{\bullet}$ denote $L^2$ norm of a function on the spacetime
$\RR^3.$

\begin{definition}
Let $\D$ denote the space of test
functions $\mathcal{C}_c^\infty(\RR_t\times (\RR^2_x \backslash \{0\}));$ let
$\D'$ denote the dual space and let $\E'$ denote the elements of $\D'$
that are compactly
supported in $\RR^3.$  Let $\D(U)$ denote those test function
supported in $U.$ Let $\schwartz$ denote Schwartz
functions supported in $x \neq 0$ and $\schwartz'$ their dual.  For
$\mathcal{X}$ any of the above spaces of distributions, we let
$\mathcal{X}_k$ denote the subspace of distributions that are annihilated by $\pa_\varphi-ik.$
\end{definition}

We may now define Hilbert spaces adapted to our problem.
\begin{definition}
Fix $k \in \ZZ.$  Let $\hilb_k^1$ be the closure of the the $S^1_\varphi$-equivariant
  test functions $\D_k$ with respect to the squared norm
  $$
\norm{\bullet}_{\hilb_k^1}^2= \norm{\bullet}^2+\norm{\pa_t \bullet}^2+ \norm{\pa_r \bullet}^2+\norm{r^{-1}(\C \pa_t+ik) \bullet}^2.
$$
Let $\hilb_k^{-1}$ denote the dual space with respect to the  $L^2$ inner product. 
\end{definition}
\begin{remark}
Away from $r=0,$ $\hilb_k^1$ is just equivariant functions in
$H^1;$ at $r=0,$ though, membership in this space entails subtly different estimates than
$H^1$ regularity; for instance, smooth compactly supported functions
of $(t,r)$ times $e^{ik\varphi}$ are not in $\hilb_k^1.$

We can (and will) identify $\hilb_k^{-1}$ with a space of equivariant distributions.
\end{remark}



\begin{lemma}\label{lemma:compact}
For any $K$ compact and $S_\varphi^1$-invariant, the inclusion 
$$
\hilb_k^1\cap \E'(K) \hookrightarrow L^2(K)
$$
is compact.
\end{lemma}
\begin{proof}
  We remark that elements of the space
  $$
e^{-ik\varphi} \hilb_k^1=\{e^{-ik\varphi} u: u \in \hilb_k^1\}
$$
are rotation-invariant in $\varphi,$ hence annihilated by $\pa_\varphi$
or even by $r^{-1} \pa_\varphi$.  Thus if $u_j$ are a sequence of
elements in the unit ball in $\hilb_k^1,$ supported in $K$ then
$$
v_j \equiv e^{-ik\varphi} u_j
$$
enjoy the same support property and satisfy
$$
\norm{v_j}^2+ \norm{\pa_t v_j}^2+ \norm{\pa_r v_j}^2+
\norm{r^{-1}\pa_\varphi v_j}^2 \leq 1,
$$
where the last term on the LHS is of course zero.  Recognizing that
the LHS is now the usual $H^1$ norm, we see that $L^2$-convergence of
a subsequence follows
from compact embedding of $H^1 \cap\E'(K)$ in $L^2.$
  \end{proof}

In discussing weak solutions to $(\Box_k +m^2) u=0$ we must be careful about
behavior near $r=0,$ since in fact $\Box_k$ does not a priori map
even $\mathcal{C}_c^\infty(\RR^3)$ to distributions,
owing to the singularity at $r=0.$  Hence in discussing distributional
solutions to $(\Box_k +m^2)u=f$ we will mean weak solutions in the following
sense.
\begin{definition}\label{definition:weaksol}
For $ u \in L_k^2$, and $U \subset \RR^3$ open and $S^1_\varphi$-invariant, we define $(\Box_k+m^2) u=f$
on $U$ if $$\ang{u,(\Box_k+m^2) \phi}=\ang{f,\phi}$$ for all test
functions $$
\phi \in \D_k(U).
$$
\end{definition}

(We will use the same definition of weak solution in dealing with the modified
operators $P,P^*$ defined below.)
  
  With a notion of solutions and appropriate Sobolev spaces in hand,
  we can now state our main theorem.  {For $q \in T^*\RR^3,$ let $\Phi^s(q)$ denote the
    Hamilton flow (with
  Hamiltonian given by the
  principal symbol of $\Box_k+m^2$) with parameter $s$ starting at $q;$ note that the
    $t$ variable may be increasing or decreasing along the flow
    according to the sign of its dual variable.
For a set
  $\Omega \subset T^*\RR^3$ let $$\Phi_+(\Omega)=\bigcup_{s \in\RR}
  \{\Phi^s(q): q \in \Omega,\ t(\Phi^s(q))\geq t(q)\}$$ denote the
  \emph{forward-in-time} flowout; projected to the base, this is a
  forward-in-time motion along
  radial geodesics for the cosmic string metric.
Finally, let $\Sigma$ denote the
  characteristic set of $\Box_k+m^2,$ intersected with that of
  $\pa_\varphi-ik,$ i.e., the radial part of the light cone.  (See
  Section~\ref{section:propagation} below for details on the
  Hamiltonian dynamics.)}
\begin{maintheorem}
  Given $m \in \RR,$ $R_0 >0,$ an $S^1_\varphi$-invariant compact set $K \subset
  \RR^3,$ and $f \in \hilb_k^{-1}$ with $\supp f \subset \{r<R_0\},$ there
  exists $u \in L_k^2(K)$ such that
  $$
(\Box_k+m^2) u=f\ \text{ on } K^\circ
$$
{and such that $\WF u\backslash \WF f\subset \Phi_+(\WF f \cap \Sigma).$}

If, additionally, $f \in L_k^2(\RR^3),$ then we further conclude that
$$
u \in L_k^2(K) \cap H^1_{\loc}(K^\circ \backslash\{r=0\}).
$$

The forward solution $u$ is unique modulo an element of $L_k^2 \cap
\CI(K^\circ \backslash \{r=0\}).$
\end{maintheorem}

We begin with a unique continuation theorem that rules out solutions
that are supported in the hyperbolic region.
\begin{lemma}\label{lemma:ellipticvanishing}
Let $u \in \schwartz_k'(\RR^3)$ and assume $(\Box_k+m^2)u=0$ in $\{r \in
I\}$ where $I\subset (0,\infty)$ is an open interval containing $r=\C.$ If $u(t,r)=0$ for
$r \in (0, \C) \cap I$  then $u \equiv 0$ on $\{r\in I\}.$
\end{lemma}
Note that the same proof as that given here shows that a nontrivial solution to the Tricomi equation $y\pa_x^2+\pa_y^2$ near $y=0$ cannot identically vanish in the elliptic region $y>0.$
\begin{proof}
  Let $\hat u(\lambda,r)$ denote the partial Fourier transform in $t.$  Then
\begin{equation}\label{FTed}
\big[ -\big(1-\frac{\C^2}{r^2}\big) \lambda^2 - \frac 1{r^{2}}
(r\pa_r)^2 + \frac{2\C k\lambda}{r^2}+\frac{k^2}{r^2}+m^2\big] \hat u=0,\ r \in I.
\end{equation}
Now replace $\hat u$ by $\hat u_\phi\equiv u *\phi$ for an arbitrary test function $\phi(\lambda)$ to obtain a smooth (in $\lambda$) solution to the above equation.
The Picard--Lindel\"of theorem applied to the ODE \eqref{FTed} means that for
any fixed $\lambda,$ if $\hat u_\phi(\lambda,r)=0$ for $r \in I \cap (0,\C),$  then it is
identically zero.  Thus, $\hat u_\phi=0$ identically for $r \in I$.  Since $\phi$
was arbitrary, the distribution $u$ must vanish for $r \in I.$\end{proof} 

\section{Propagation of singularities}\label{section:propagation}

We now analyze the pair of operators $(\Box_k+m^2, \pa_\varphi-ik)$ from
the perspective of microlocal analysis.
Let $(\lambda,\xi,\eta)$ denote canonical dual coordinates in
$T^*\RR^3$ to the cylindrical coordinates $(t,r,\varphi).$  The
principal symbol of $\Box_k$ (and likewise of $\Box_k+m^2$) is
$$
\sigma_2(\Box_k) = \frac 1{r^2} \C^2\lambda^2-\lambda^2 +\xi^2,
$$
hence the 
Hamilton vector field of $\sigma_2(\Box_k)$ is
$$
-2\lambda\big(1-\frac{\C^2}{r^2} \big)\pa_t +2 \xi \pa_r+\frac 2{r^3}
\C^2 \lambda^2 \pa_\xi.
$$
Meanwhile the operator $\pa_\varphi-ik,$ on whose nullspace we work,
is globally elliptic except at $\eta=0,$ hence we need only concern
ourselves with this region of phase space.  The system
$(\Box_k+m^2,\pa_\varphi-ik)$ is then elliptic for $r<\C.$

In the following, let $\Sigma$ denote the \emph{joint} characteristic
set of $(\Box_k+m^2,\pa_\varphi-ik),$ hence {the subset of the
complement of the zero-section} given by $\{\sigma_2(\Box_k)=0\}
\cap\{\eta=0\}$.  By standard elliptic regularity, for $u \in\schwartz'_k$,
$$\WF u \subset \WF \big((\Box_k+m^2) u\big)\cup \Sigma,$$ at least over $r>0.$  (We will
develop an elliptic estimate below that is valid down to $r=0.$)
{\begin{remark}
   Our system $(\Box_k+m^2, \pa_\varphi-ik)$ changes type abruptly across
  the hypersurface $\{r=\C\}\subset \RR^3,$ hence the projection to
  the base of $\Sigma$ has a boundary at $r=\C.$  By contrast, upstairs in the cotangent bundle
  $$
\Sigma=\{(r^2-\C^2) \lambda^2=r^2 \xi^2,\ \eta=0\}
$$
is nonetheless a smooth conic submanifold of $T^*\RR^3.$  It is only
the projection to the base that is singular.\end{remark}}

\begin{remark}\label{remark:flow}
For later use, we note that on $\Sigma,$ vanishing of
$\sigma_2(\Box_k)$ gives
$$
\lambda^2(1-\C^2/r^2)= \dot{r}^2/4,
$$
with dot denoting derivative along the Hamilton flow.
  Hence along bicharacteristics with flow parameter $s,$
  $$
r(s)^2= \C^2+(2\lambda s+ \text{const})^2,
$$
i.e.\ $r \to \infty$ (in a monotone fashion) as $s \to \pm \infty.$  Meanwhile, $\dot{t}
=-2\lambda(1-\C^2/r^2)$ yields $\abs{\dot{t}}\geq \abs{\lambda}$ for
$r$ sufficiently large, i.e.\ $t$ is strictly monotone along the flow
as $s \to \pm\infty.$

The integral curves are singular in the base (i.e., $t,r$ variables)
when $r=\C,$ and $t$ is stationary there, since
$dt/dr=-\lambda(r^2\xi)^{-1}(r^2-\C^2).$ But the curves are smooth in the
cosphere bundle: such points are not radial points since $\dot
\xi=2r^{-3} \C^2 \lambda^2 \neq 0.$ (Note that $\lambda
\neq 0$ on $\Sigma$.)
\end{remark}

Given a fixed compact $K \subset \RR^3$ and $R_0 \in\RR$, choose
$R>\max\{\C, R_0\}$ such that $K
\subset \{r<R\}.$  Note that over $r>R,$ the characteristic set
$\Sigma$ separates into four components corresponding to choosing
$\lambda \gtrless 0,$ $\xi \gtrless 0.$  Let $$\Sigma_- \equiv
\Sigma\cap \{r>\C\}
\cap \{\sgn \lambda=\sgn \xi\}.$$
{We now construct {$W \in \Psi^2(\RR^3)$} enjoying the following properties:}
\begin{enumerate}
  \item
For $r>R+1,$ $W$ is elliptic on $\Sigma_-.$
\item
Where $\sigma_2(W) \neq 0,$
  $$
\sgn \sigma_2(W)=-\sgn \lambda.
$$
\item
$\proj_\bullet \supp \kappa(W) \subset \{r>R\},$ where $\proj_\bullet$
is projection to the left or right factor ($\bullet=L$ or $R$) and
$\kappa$ denotes Schwartz kernel.
\item $[\pa_t,W]=[\pa_\varphi,W]=0.$
\end{enumerate}
{To produce such an operator, we begin by choosing $R$ such that $1-\C^2/R^2>9/10.$
Let $\varrho, \psi, \chi$ be smooth 
functions such that
\begin{itemize}
\item $\varrho(s)= 1$ on $(3/4, 5/4)$ and is supported on $(2/3, 4/3).$
  \item $\psi(s)=1$ on $(-1/10, 1/10)$ and is supported on $(-1/5,1/5).$
    \item $\chi(r)$ is supported on $(R,\infty)$ and equals $1$ on $(R+1,\infty).$
  \end{itemize}
Let $w$ denote the
  homogeneous $2$-symbol
  $$
w = -\sgn(\lambda) \lambda^2 \varrho(\xi/\lambda) \psi(\abs{\eta/\lambda})\chi(r).
$$
Let $W_0$ denote the Weyl quantization of $w,$ and let
$$
W=\widetilde \chi W_0 \widetilde \chi
$$
where $\widetilde\chi(r)=1$ on $\supp \chi(r),$ and $\widetilde\chi$ is supported in
$(R,\infty).$  The cutoffs $\widetilde\chi$ enforce the support
properties of the kernel of $W.$  The ellipticity property follows
from the fact that on $\Sigma_-\cap \{r>R\},$  $\xi^2 \in
(9\lambda^2/10,\lambda^2)$ and $\eta=0,$ hence the $\varrho$ and $\psi$ cutoffs
equal $1,$ and $w=-\sgn(\lambda) \lambda^2$ on
this set.}

We will consider solutions to $Pu=f$ for the operator
$$
P \equiv \Box_k+m^2-iW.
$$

\begin{remark}\mbox{}
\begin{enumerate}\item The set $\Sigma_-$ is \emph{incoming} in forward time in the sense that under bicharacteristic flow, $dr/dt<0$ there.
\item
  We will prove a number of preliminary results that hold equally well
for the operators $P$ and $P^*,$ hence we let $P^{(*)}$ denote either
of these operators. 
\item By ellipticity of $(P^{(*)},\pa_\varphi-ik)$ on
  $\{r>R+1\}\cap \Sigma_-,$
  $$(\WF u \backslash \WF P^{(*)} u) \cap
  \{r>R+1\}\cap\Sigma_-=\emptyset$$ for $u \in \schwartz'_k.$\end{enumerate}
\end{remark}

  \begin{lemma}
    The operators $P,P^*$ enjoy the following mapping property:
    $$
P^{(*)}: \hilb_k^1\to \hilb_k^{-1}.
    $$
  \end{lemma}
  \begin{proof}
    For test functions $\phi,\psi\in \D$ with $\supp \phi \subset \{r<R\},$
    $$
{\ang{P^{(*)}\phi,\psi}} =\ang{r^{-1} (\C\pa_t+ik)\phi, r^{-1}
  (\C\pa_t+ik)\phi}-\ang{\pa_t\phi,\pa_t\psi} +
\ang{\pa_r\phi,\pa_r\psi}+ m^2 \ang{\phi,\psi}.
$$
Applying Cauchy--Schwarz to each term on the RHS, we may estimate it
by a multiple of $\norm{\phi}_{\hilb_k^1} \norm{\psi}_{\hilb_k^1},$ hence
the mapping property follows.  For test functions with support in
$r>R/2$, on the other hand, the estimate simply follows from
boundedness of second order differential operators from $H^1 \to
H^{-1}$ since the norm on $\hilb_k^1$ is equivalent to the $H^1$ norm away from $r=0.$

{Thus, choosing a cutoff $\chi(r)$ equal to $1$ on
$[0, R/2)$ and supported in $[0, 3R/4),$ given
any $\phi \in \D$ we split
$$
\phi=\chi \phi+(1-\chi) \phi.
$$
The operations of multiplication by $\chi,$ $1-\chi$ are
bounded on $\hilb_k^1$ since $\chi=1$ near the origin.  Thus
$$
P^{(*)}\phi=P^{(*)}(\chi\phi)+P^{(*)}((1-\chi) \phi)
$$
is bounded in $\hilb_k^{-1}$ by a multiple of $\norm{\phi}_{\hilb_k^1}$
by applying the foregoing results to $\chi\phi$ and $(1-\chi)\phi$.}
\end{proof}

The virtue of our construction of $W$ is that owing to our choice of
signs for $W,$ regularity {for solutions to the equation $Pu=f$} propagates
\emph{forward} along {null bicharacteristics of $\Re\sigma_2(P)$} in the hyperbolic region for
$\lambda<0$ (since $\sigma_2(W)\geq 0$) and \emph{backward} for
$\lambda>0$ (since $\sigma_2(W) \leq 0$);  we refer the reader to
\cite[Section 2.5]{Va:13} for a proof.  Hence, since $\dot
t=-2\lambda (1-\C^2/r^2)$ along the flow, and this has the same sign as
$-\lambda$ in the hyperbolic region, on every component of the
characteristic set, regularity propagates \emph{forward in time} {for
the operator $P.$  Of course, when we consider $P^*,$ the sign of the
$W$ term is reversed, and the reverse phenomenon therefore takes place:
regularity propagates \emph{backward in time} instead.  We will use
both of these propagation results in what follows: that for $P^*$ to
obtain solvability of the equation $Pu=f$ and that for $P$ to
constrain the wavefront set of the resulting distribution $u.$}

\begin{lemma}\label{lemma:nonull}
If $u \in \E'_k(\RR^3) \cap L^2(\RR^3)$ and $P^{(*)}u=0$ then $\supp u
\subset \{r\geq R\}.$
\end{lemma}
\begin{proof}
This follows from our uniqueness results (cf.\ \cite{Ba:02}), since
vanishing on an open subset of the elliptic set $\{r \in (0, \C)\}$ implies global
vanishing on this set, by unique continuation for elliptic operators, and in turn implies
global vanishing on $\{ r \in (0, R)\}$ by
Lemma~\ref{lemma:ellipticvanishing}.
\end{proof}
Of course for $r\geq R,$ since $P^{(*)} \neq \Box_k+m^2,$ we cannot conclude vanishing
anymore.

\begin{lemma}\label{lemma:Kprime}
  Given $K\subset \RR^3$ compact and $S^1_\varphi$-invariant, and $R\gg 0$ chosen as above, there
  exists $K'\supset K,$ also compact and invariant, such that for any
  radial null bicharacteristic $\gamma(s)\subset\Sigma$ of
  $\sigma_2(\Box_k)=\Re \sigma_2(P^{(*)})$ with
  $\pi\gamma(0) \in K,$ there exists $s_0$ such that 
  \begin{enumerate}
  \item $\pi\gamma(s_0) \in K'$
  \item $t(\gamma(s_0))<t(\gamma(0)),$
    \item $r(\gamma(s_0))>R+1$
\item $\sgn d (r\circ \gamma) /ds\rvert_{s_0}=-\sgn d (t\circ \gamma)/ds\rvert_{s_0}.$\end{enumerate}
  \end{lemma}
That is to say, every null bicharacteristic starting in $K$ hits the
elliptic set of $W$ in backward time, within the set $K',$ and does so
at ``incoming points'' where $dr/dt<0.$  (See Figure~\ref{fig:1}.)   The lemma
follows directly from Remark~\ref{remark:flow}, since for one choice
of the sign, as the
flow parameter $s \to \pm\infty$ we have $r \to +\infty,$ $t\to-\infty,$
with both functions monotone.  Hence we may simply take $K'=[0, R+2]_r \times
[-T,T]_t \times S^1_\varphi$ with $T$ sufficiently large that every
bicharacteristic escapes to $r>R+1$ in backward time no less than $-T.$

\begin{figure}
  \begin{center}
    \begin{tikzpicture}[scale=2.5]
      \def\b{-2};
    \draw[thick,->](-1,\b)--(-1,1);
    \draw[thick,->](-1,0)--(1,0);
    \draw[blue,thick](0.8,\b)--(0.8,1);
        \draw[blue,thick](-0.9,\b)--(-0.9,1);
    \node[align=left] at (1.1, 0){$r$};
    \node[align=left] at (-1.1, 1.05){$t$};
    \node[color=blue, align=left] at (1.3,0.7){$r=R+1$};
        \node[color=blue, align=left] at (-0.6,0.7){$r=\C$};
    \filldraw[opacity=0.3, color=red] (-1,-0.4) -- (-0.4, -0.4) --
    (-0.4,0.4) -- (-1,0.4);
    \node at (-0.6,0.2){$K$};
        \filldraw[opacity=0.3, color=yellow] (-1,\b+0.1) -- (1.2,
        \b+0.1) -- (1.2,0.5) -- (-1,0.5);
            \node at (0,0.5*\b-.4){$K'$};
        \draw[densely dashed] (-0.8, -0.2) .. controls  (0, -0.2) and (0.9, \b+0.4) .. (1.1, \b+0.2);
\end{tikzpicture}
\end{center}
\caption{\label{fig:1} The sets $K$ and $K'$ with a sample (projected) bicharacteristic with one
  end in $K$ arriving in backward time in the set $K'$ inside $r>R+1$
  (dashed curve).}
\end{figure}
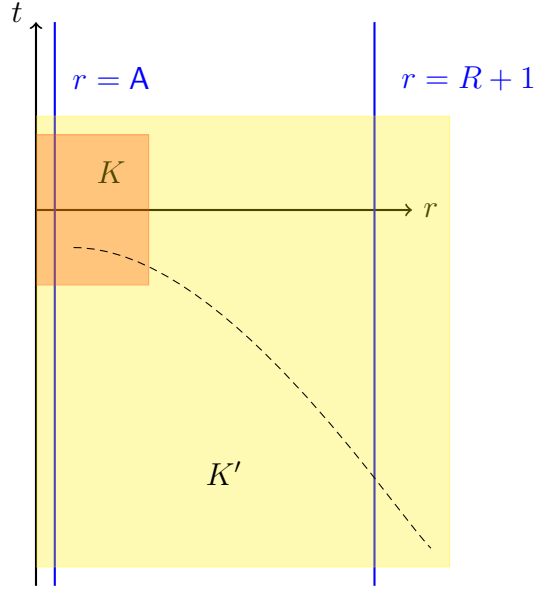

We additionally need an elliptic estimate valid down to $r=0$: for $u$
supported in $\{\ep\leq r \leq \C-\ep\},$ standard elliptic estimates
apply, but we will need uniformity down to $r=0$ as well.  We begin
with a coercivity estimate.

\begin{lemma}\label{lemma:ellipticestimate}
  For all $\phi\in \D_k$
supported in $\{r \in (0, \C/4)\},$ 
\begin{equation}\label{pairing} 
  \norm{\phi}^2_{\hilb_k^1} \lesssim
  \norm{\phi}^2_{L^2} + \ang{P^{(*)} \phi, \phi}.
\end{equation}
    \end{lemma}
    
  \begin{proof}
Pairing $P^{(*)} \phi$ with $\phi$ gives
    $$
 \ang{P^{(*)}\phi,\phi}=\norm{r^{-1}(\C \pa_t +ik) \phi}^2-\norm{\pa_t
   \phi}^2+ \norm{\pa_r\phi}^2+m^2 \norm{\phi}^2.
$$
Since for $r<\C/4$ \begin{align*} \frac 14 r^{-2}(\C\lambda+k)^2-\lambda^2 &\geq \big(2 \lambda+\frac{2k}\C\big)^2-\lambda^2\\ &= \lambda^2 +2\lambda^2 +\frac {8 \lambda k }{\C} + \frac{4 k^2}{\C^2} \\ &=\lambda^2 +2 \big( (\lambda+\frac{2k}{\C})^2 -\frac{2k^2}{\C^2} \big),
                  \end{align*}
                  Fourier transforming $t \to \lambda$ and using Plancherel gives the elliptic estimate
                  $$
\frac 14 \norm{r^{-1}(\C \pa_t +ik) \phi}^2-\norm{\pa_t \phi}^2 \geq \norm{\pa_t\phi}^2-\frac{4k^2}{\C^2} \norm{\phi}^2.
                  $$
                  Thus
$$
\ang{P^{(*)} \phi,\phi} \geq \frac 34 \norm{r^{-1}(\C \pa_t+ik) \phi}^2+
                    \norm{\pa_t \phi}^2+\norm{\pa_r\phi}^2 +
                    \norm{\phi}^2-\big( 1+
                    \frac{4k^2}{\C^2})\norm{\phi}^2
$$
and we have obtained the desired estimate.
\end{proof}

  \begin{lemma}\label{lemma:elliptic}
For $u \in \hilb_k^1,$ supported in $r<\C/4,$
\begin{equation}\label{ellipticestimate}
\norm{\phi}_{\hilb_k^1}\leq C \norm{\phi}_{L^2} + C\norm{P^{(*)} \phi}_{\hilb_k^{-1}}.
\end{equation}
  \end{lemma}
  \begin{proof}
Apply Cauchy-Schwarz to the estimate
\eqref{pairing} and use the density of test functions.
                \end{proof}
    \begin{remark}
The hypothesis that $\phi \in \hilb_k^1$ may \emph{not} be dispensed
with in the preceding lemma.  One might hope that a weaker a priori assumption, such as
$\phi \in L^2,$ $P^* \phi \in \hilb_k^{-1},$ might \emph{imply} $\phi
\in \hilb_k^1,$ but this is not so.  For instance, as noted by Carrillo
\cite{Ca:21}, take $m=0$ and fix $\lambda$ such that $\C\lambda+k \in (-1,0)$ and consider
$$
\phi(t,r,\varphi) = \chi(r) e^{i\lambda t}e^{ik\varphi} J_{\C \lambda+k} (\lambda r),
$$
with $\chi(r)$ a cutoff function equal to $1$ on $[0,\C/2) $ and supported in
$r< \C.$ This satisfies
$$
P^* \phi = [\Box_k, \chi]e^{i\lambda t} e^{ik\varphi} J_{\C \lambda+k} (\lambda r),
$$
which is in $\CI(\RR^3),$ supported away from $r=0,$ hence certainly
$P^* \phi$ locally lies in ${\hilb_k^{-1}}.$  Owing to our choice of
$\C\lambda+k,$ we have moreover arranged that $\phi \in L^2_{\loc}.$ But $\pa_r\phi$
is not in $L^2_{\loc}$ near $r=0.$  Thus, local finiteness of the RHS
of the
estimate \eqref{ellipticestimate} certainly cannot guarantee that $u \in \hilb^1_{k,\loc}.$  To see the
global failure, we instead take a superposition of these examples.
Take $\C>0$ for simplicity of notation and construct
$$
\phi(t,r,\varphi)=e^{ik\varphi} \int \zeta(\lambda) \chi(r) e^{i\lambda t} J_{\C
  \lambda+k} (\lambda r) \, d\lambda,
$$
where $\zeta(\lambda)$ is a smooth function compactly supported in
$$
-\C^{-1}(3/4+k)<\lambda<-\C^{-1} (1/4+k),
$$
i.e., such that $$\nu \equiv \C\lambda+k\in (-3/4,-1/4)\text{ for }\lambda \in \supp
\zeta.$$  (We will further specify $\zeta$ below.)  Now again $P^*\phi$
is compactly supported in $r>0$ but this time it is also Schwartz in
$t,$ hence $P^* \phi \in L^2.$

On the other hand, applying Taylor's theorem to the family of analytic
functions of $r$ given by
$$
r^{-\nu}J_{\nu} (\lambda r),\ \nu \in (-3/4,-1/4)
$$
yields
$$
J_{\nu} (\lambda r)=f_0(\nu)r^\nu+O(r)
$$
with the remainder term estimated uniformly for $\nu \in (-3/4,1/4)$;
here $f_0(\nu)$ {is a smooth
(indeed, locally analytic) function of $\nu=\C\lambda+k,$
nonvanishing for $\nu\notin \ZZ$ (and $\lambda \neq 0$).}  Likewise
$$
\pa_r J_{\nu} (\lambda r)=f_1(\nu)r^{\nu-1}+O(1)
$$
for some other locally analytic function $f_1(\nu)$ {nonvanishing on
$\nu \in \RR\backslash \ZZ,\ \lambda \neq 0;$ again we have
uniform remainder bounds}. Hence if we choose $\zeta(\lambda)$ so that
the product
$$
\zeta(\lambda) f_1(\C\lambda+k) 
$$
is a nonnegative cutoff function equal to $1$ in
$-\C^{-1}(2/3+k)<\lambda<-\C^{-1} (1/3+k),$
then applying Plancherel in $\lambda$ shows
that on the one hand, $\phi \in L^2.$ On the other hand, it also
yields
$$
\norm{\pa_r \phi}^2 \geq \C^{-1} \int_0^{\C/2} \int_{-2/3}^{-1/3} (r^{2(\nu-1)}+O(1)) \, d\nu\,r
\, dr,
$$
which diverges.
\end{remark}

\section{Proof of the Theorem}                                                                                                                                                                  

We now follow an approach similar to that of Payne \cite{Pa:98} to
show, using Duistermaat--H\"ormander style microlocal energy
estimates, that forward parametrices exist semiglobally. 
In particular, we now describe the crucial propagation estimate; from here our argument hews closely to 
\cite[Theorem 6.3.1]{DuHo:72}.

In what follows, the constant $C$ will be allowed to change from line
to line.
\begin{proposition} For $\phi \in\hilb_k^1 \cap \E'(K'),$
\begin{equation}\label{fredholm1}
\norm{\phi}_{\hilb_k^1} \leq C \norm{P^* \phi}_{L^2}+ C \norm{\phi}_{L^2}.
\end{equation}
\end{proposition}
\begin{proof}
  Let $\chi(r)$ be a cutoff function equal to $1$ on $r<\C/8$ and supported in $(-\infty, \C/4).$ Then applying Lemma~\ref{lemma:elliptic} to $\chi \phi$ gives
  $$
\norm{\chi\phi}_{\hilb_k^1}\leq C \norm{\phi}_{L^2} + C\norm{\chi P^* \phi}_{\hilb_k^{-1}} + C \norm{[P^*,\chi]\phi}_{\hilb_k^{-1}}.
  $$
  Since $[P^*, \chi]$ is an operator of order $1$ with smooth
  coefficients, supported away from $r=0$ (hence $\hilb_k^{-1}$ locally agrees
  with $H^{-1}$)
  $$
 \norm{[P^*,\chi]\phi}_{\hilb_k^{-1}} \lesssim \norm{\phi}_{L^2}
 $$
 and we conclude a fortiori (since $L^2 \subset \hilb_k^{-1}$) that
 $$
\norm{\chi\phi}_{\hilb_k^1} \leq C \norm{P^* \phi}_{L^2}+ C \norm{\phi}_{L^2}.
$$
It now suffices to additionally show that
\begin{equation}\label{oneminuschi}
\norm{(1-\chi)\phi}_{H^1} \leq C_1 \norm{P^* \phi}_{L^2}+ C_2 \norm{\phi}_{L^2},
\end{equation}
where we have switched to using the ordinary $H^1$ Sobolev norm, since it
agrees with the $\hilb_k^1$ norm on the hyperbolic region.
To show \eqref{oneminuschi}, let $q_0=(t_0,r_0,\lambda_0,\xi_0)
\in\pi^{-1}(\supp(1-\chi) \phi) \subset S^*X.$ If $q_0 \in
\elliptic(P^*),$ then for $A_{q_0}\in \Psi^0_c(X)$ microsupported
sufficiently close to $q_0$ (chosen, for use later, with nonnegative principal symbol)
$$
\norm{A_{q_0} \phi}_{H^1} \leq C \norm{\phi}_{L^2}+ C \norm{P^* \phi}_{H^{-1}}
$$
by standard elliptic estimates.  (This estimate is stronger than
needed, owing to the $H^{-1}$ norm on the RHS, but the weaker
estimate in the statement of the proposition will be as good as we can obtain on the hyperbolic set.)   On the other hand, if $q_0$ is in the
characteristic set $\Sigma,$ then either $\lambda >0$ or
$\lambda<0$ along the whole null bicharacteristic $\gamma$ of $\Re
\sigma_2(P^*)$ through
$q_0,$ since $\lambda=\lambda_0$ is conserved under the flow (owing to
$t$-independence of the symbol), {hence by our choice of the sign of
$\sigma_2(W)$ we have propagation of regularity \emph{backwards in
  time} along null bicharacteristics of $\Re \sigma_2(P^*)$ by the results of \cite[Section
2.5]{Va:13}. 
  Since the flow eventually leaves $K'\supset \supp \phi$ as $t
  \to +\infty,$}
$$
\norm{A_{q_0}  \phi}_{H^1} \leq C \norm{P^* \phi}_{L^2}+ C\norm{\phi}_{L^2}
$$
(cf.\ \cite[Equation 2.18]{Va:13}).
Piecing together these estimates for a finite cover of $K'\cap \supp (1-\chi)$ by
$\elliptic
A_{q_1},\dots \elliptic A_{q_N}$ and invoking elliptic regularity for
$\sum A_{q_j}$ yields the estimate \eqref{oneminuschi}.
\end{proof}

Now we claim that the second term on the RHS of \eqref{fredholm1} can
be dropped if we restrict ourselves to a finite codimension subspace
of $\hilb_k^1\cap \E'(K')$: we let
$$
N(P^*) = \{u \in \hilb_k^1 \cap \E'(K'): P^* u=0\}.
$$
The space $N(P^*)$ is finite dimensional, since \eqref{fredholm1}
implies that on this space $\norm{\phi}_{\hilb_k^1} \lesssim
\norm{\phi}_{L^2},$ hence the unit ball of $N(P^*)$ in the $L^2$
topology is compact, by Lemma~\ref{lemma:compact}.
Let $N(P^*)^\perp$ denote the orthocomplement of this
finite-dimensional space in $L^2_k.$
\begin{lemma}
\begin{equation}\label{fredholm2}
\norm{\phi}_{\hilb_k^1} \leq C \norm{P^* \phi}_{L^2},\quad \phi \in
\hilb_k^1 \cap \E'(K') \cap N(P^*)^\perp.
\end{equation}
\end{lemma}
\begin{proof} If \eqref{fredholm2} did not hold, there would exist $\phi_j \in \hilb_k^1\cap N(P^*)^\perp$ with support in $K'$ such that
$$
\norm{\phi_j}_{\hilb_k^1}=1,\ \norm{P^* \phi_j}_{L^2}\to 0.
$$
Extracting a weakly convergent subsequence in $\hilb_k^1,$ hence
strongly convergent in $L^2$ by Lemma~\ref{lemma:compact}, we get
$\phi_j \to \psi \in \hilb_k^1\cap N(P^*)^\perp$ with convergence in the $L^2$ sense.
Recall that $P^*: \hilb_k^1\to \hilb_k^{-1}$ is continuous, so $P^* \psi$
is the weak limit in $\hilb_k^{-1}$ of $P^*\phi_j.$ On the other hand $P^* \phi_j
\to 0$ in $L^2,$ hence a fortiori in $\hilb_k^{-1},$
so in fact $P^* \psi=0$. Consequently, $\psi\in \hilb_k^1 \cap N(P^*)$, i.e.\
$\psi=0.$  Thus $\phi_j \to 0$ in $L^2,$ and \eqref{fredholm1} for the
sequence $\phi_j$ reads 
\begin{equation}
\norm{\phi_j}_{\hilb_k^1} \leq C \norm{P^* \phi_j}_{L^2}+ C \norm{\phi_j}_{L^2} \to 0,
\end{equation}
contradicting the assumed normalization of the LHS.\end{proof}

As a result of \eqref{fredholm2}, if $f \in N(P^*)^\perp,$ the map
$$
T\colon P^* \phi \mapsto \ang{\phi,f}
$$
is well-defined on the range of $P^*$ on the test functions
$\D_k((K')^\circ)$ considered as a subset of
$L_k^2(K')$; \eqref{fredholm2} and the dual pairing of $\hilb_k^1$ and
$\hilb_k^{-1}$ yields
$$
\abs{T P^* \phi} \lesssim \norm{P^* \phi}_{L^2}\norm{f}_{\hilb_k^{-1}}.
$$
We now extend the map to the whole of $L_k^2(K')$ by Hahn--Banach. 
The Riesz Lemma implies the existence of $u \in L_k^2(K')$ with $TP^*\phi  =\ang{P^*\phi,u},$ hence
$$
\ang{\phi,f}=\ang{\phi, Pu}
$$
for all test functions $\phi$ supported on
$(K')^\circ \backslash \{r=0\}.$ Hence $u$ solves $Pu=f$ on
$(K')^\circ$ (in the weak sense specified by
Definition~\ref{definition:weaksol}).  Of course, we were restricted to
$f \in N(P^*)^\perp$ in making this construction. By Lemma~\ref{lemma:nonull}, though, elements of
$N(P^*)$ are supported entirely in $r\geq R,$ hence if we restrict to
$\supp f \subset \{r<R_0\},$ having chosen $R>R_0$ ensures that $f \in N(P^*)^\perp$ and the solvability result
applies.

To see that $u$ has the desired wavefront properties on $K,$ we now
bring to bear Lemma~\ref{lemma:Kprime}.  Owing to our choice of the
sign of $W,$ recall that regularity propagates forward \emph{in time}
under
bicharacteristic flow (away from $\WF f$).  {Ellipticity guarantees
both $\WF u \subset \WF f \cup \Sigma$ and also $\WF u \cap
\{r>R+1\}\cap \Sigma_-=\emptyset.$ Since every
bicharacteristic passing through $K$ reaches this latter set inside
$\pi^{-1} K'$ in backward time, no point may be in $\WF u\backslash
\WF f$ whose backward-in-time flow does not hit $WF f.$  In other words,
$(\WF u\backslash \WF f)\cap \pi^{-1}K$ is contained
in the \emph{forward} flowout of $\WF f\cap \Sigma,$ as desired.}
Meanwhile, $u$ does solve our original equation $(\Box_k+m^2) u=f$ on $K,$
since $W=0$ on $K$ (recall $K \subset \{r<R\}$).

In case we have the increased regularity $f \in L_k^2,$ we conclude that
$u \in H^1_{\loc}$ by propagation of singularities in the hyperbolic
region; on the elliptic set, we can even do better if desired ($u \in
H^2_{\loc}$).

If $u_1, u_2$ are both forward solutions to $(\Box_k+m^2) u_\bullet=f,$ then
$u_1-u_2 \in \CI$ microlocally on the elliptic set of $P,$ by elliptic
regularity.  Elsewhere we obtain $u_1-u_2 \in \CI$
by propagation of singularities, since
$\WF u_\bullet \cap \{r>R+1\}\cap \Sigma_-=\emptyset$ and
$(\Box_k+m^2)(u_1-u_2)=0.$  The uniqueness assertion of the theorem follows.
\qed

  \bibliographystyle{abbrv} 
\bibliography{CosmicStrings}

\end{document}